\documentclass[11pt]{amsart}
\usepackage{amssymb}

\input epsf.tex

\input xy
\xyoption{all}

\newtheorem{thm}{Theorem}[section]

\newtheorem{lem}[thm]{Lemma}
\newtheorem{prop}[thm]{Proposition}
\theoremstyle{definition}

\theoremstyle{remark}

\numberwithin{equation}{section}

\newcommand{\delete}[1]{} % Comment out text.

\newcommand{\ben}{\begin{enumerate}}
\newcommand{\een}{\end{enumerate}}

\def\Homeo{{\mathrm{Homeo}}\,}

\def\QED{\nobreak\quad\ifmmode\roman{Q.E.D.}\else{\rm Q.E.D.}\fi}

\newcommand{\Bcal}{\mathcal{B}}

\newcommand{\Dcal}{\mathcal{D}}

\newcommand{\Fcal}{\mathcal{F}}
\newcommand{\Gcal}{\mathcal{G}}

\newcommand{\Pcal}{\mathcal{P}}
\newcommand{\Qcal}{\mathcal{Q}}
\newcommand{\Rcal}{\mathcal{R}}
\newcommand{\Scal}{\mathcal{S}}

\newcommand{\C}{\mathbb{C}}
\newcommand{\N}{\mathbb{N}}

\newcommand{\Z}{\mathbb{Z}}

\newcommand{\al}{\alpha}
\newcommand{\Ga}{\Gamma}
\newcommand{\ga}{\gamma}
\newcommand{\del}{\delta}

\newcommand{\ep}{\varepsilon}
\newcommand{\sig}{\sigma}
\newcommand{\Sig}{\Sigma}
\newcommand{\la}{\lambda}

\newcommand{\tet}{\theta}
\newcommand{\Tet}{\Theta}
\newcommand{\om}{\omega}
\newcommand{\Om}{\Omega}

\newcommand{\br}{\vspace{4 mm}}

\newcommand{\rest}{\upharpoonright}

\newcommand{\ch}{\mathbf{1}}

\newcommand{\id}{\rm{id}}

\begin{document}

\title
[On two problems concerning topological centers]
{On two problems concerning topological centers}
%{On two problems presented to me at Castellon}

%Authors
\author[Eli Glasner]{Eli Glasner}
\address{Department of Mathematics,
Tel-Aviv University, Ramat Aviv, Israel}
\email{glasner@math.tau.ac.il}
%\urladdr{http://www.math.tau.ac.il/$^\sim$glasner}

\date{July 7, 2005}



\thanks{Research supported by ISF grant \# 1333/07.}

\begin{abstract}
Let $\Ga$ be an infinite discrete group and $\beta\Ga$ its
\v{C}ech-Stone compactification. Using the well known fact that
a free ultrafilter on an infinite set is nonmeasurable, we show
that for each element $p$ of the remainder $\beta\Ga \setminus \Ga$,
left multiplication $L_p:\beta\Ga \to \beta\Ga$ is
not Borel measurable.
Next assume that $\Ga$ is abelian.
Let $\Dcal \subset \ell^\infty(\Ga)$ denote the subalgebra of distal
functions on $\Ga$ and let $D=\Ga^\Dcal=|\Dcal|$ denote the
corresponding universal distal (right topological group)
compactification of $\Ga$.
Our second result is that the topological center of
$D$ (i.e. the set of $p\in D$ for which
$L_p:D \to  D$ is a continuous map)
is the same as the algebraic center and that
for $\Ga=\Z$, this center coincides
with the canonical image of $\Ga$ in $D$.
\end{abstract}

\begin{date}
{August 6, 2007}
\end{date}



\maketitle
\section{Introduction}
This short note is a direct outcome of the topology
conference held in Castell\'on in the summer of 2007.
I was presented during the conference with two problems
relating to the topological center of certain right
topological semigroups.
(A compact semigroup $A$ such that for each $p\in A$
the corresponding right multiplication
$R_p: q \mapsto qp$ is continuous is called a {\it right
topological semigroup}. The collection of
elements $p\in A$ for which
the corresponding left multiplication
$L_p: q \mapsto pq$ is continuous is called
the {\it topological center} of $A$.)
The first was a question of Michael Megrelishvili:
Given an infinite discrete group $\Ga$, which
are the elements of $\beta\Ga$ for which $L_p:\beta\Ga
\to \beta\Ga$ is a Baire class 1 map? (It is known that
the topological center of $\beta\Ga$ is exactly $\Ga$ itself,
considered as a subset of $\beta\Ga$, see e.g. \cite{HS}.)
The second problem is due to Mahmoud Filali: If $D=D(\Ga)$ is the
universal distal Ellis group of $\Ga$, identify the topological
center of $D$.

I present here a complete answer to Megrelishvili's problem,
based on the well known result that
a free ultrafilter on an infinite set is nonmeasurable,
and an answer to Filali's problem in the case $\Ga=\Z$, the group of
integers.

The interested reader is referred to \cite[chapter 1]{Gl}
and the bibliography list thereof, for more information on the
abstract theory of topological dynamics, and to
\cite{HS} for information concerning $\beta\Ga$.

I thank both Megrelishvili and Filali for addressing to me these nice
problems. I also thank the organizers of the Castell\'on meeting
for the formidable effort they put into the details of
the conference and for their warm hospitality.

\br

\section{On the center of $\beta\Ga$}

\begin{thm}\label{M}
Let $\Ga$ be an infinite discrete group and $\beta\Ga$ its
\v{C}ech-Stone compactification. For each element $p$
of the remainder $\beta\Ga \setminus \Ga$,
left multiplication $L_p:\beta\Ga \to \beta\Ga$ is
not Borel measurable.
\end{thm}

\begin{proof}
Let $\Pcal(\Ga)$ denote the collection of all subsets of $\Ga$.
Let $\Om=\{0,1\}^\Ga$ and let $\chi: \Pcal(\Ga) \to \Om$
denote the canonical map $\chi(A)= \ch_A$ for $A \subset \Ga$.
We regard $\Om$ as a compact space and let $\Bcal$ denote
its Borel $\sig$-algebra. Let $\mu=(\frac12(\del_0 + \del_1))^\Ga$
denote the product probability measure on $\Om$ and let $\Bcal_\mu$
denote the completion of $\Bcal$ with respect to $\mu$.

A well known and easy fact, which for completeness we will reproduce
below (Lemma \ref{m-filters}), is that a free ultrafilter on an
infinite set is nonmeasurable:
Viewing an element $p \in \beta\Ga
\setminus \Ga$ as an ultrafilter on $\Ga$, the collection
$\{\chi(A): A \in p\}\subset \Om$, is not $\mu$-measurable; i.e. not an
element of $\Bcal_\mu$. In particular it is not a Borel
subset of $\Om$. (In fact, it is not even Baire measurable, \cite{T}
and \cite{T1}.)

The compact space $\Om$ becomes a dynamical system when we let
$\Ga$ act on it by permuting the coordinates:
$$
\ga\om(\ga')=\om(\ga^{-1}\ga').
$$
Of course the measure $\mu$ is $\Ga$-invariant, but we will not
need this fact.
The action of $\Ga$ on $\Om$ extends to an action of $\beta\Ga$
in the natural way and we write $p\om$ for the image of
$\om \in \Om$ under $p \in \beta\Ga$. (In fact via this ``action"
$\beta\Ga$ is identified with the enveloping semigroup of the
system $(\Om,\Ga)$, see \cite[chapter 1]{Gl}.)

%For $A \subset \Ga$ we let $\ch_A:\Ga \to \{0,1\}$ denote
%the characteristic function of $A$; i.e. $\ch_A(\ga)=1$ iff
%$\ga \in A$.
For $A \subset \Ga$ and $p\in \beta\Ga$
set
$$
p\star A = \{\ga \in \Ga : \ga A^{-1} \in p\}
$$
and check that $\ga\ch_A=\ch_{\ga A}$ and $
p\chi(A)= p\ch_A =\ch_{p\star A}
=\chi(p\star A)$.
Moreover if $q\in \beta \Ga$ then
$$
pq\star A = p \star (q\star A).
$$
For convenience I sometimes
identify a subset $A \subset \Ga$ with the corresponding
element $\ch_A =\chi(A)$ in $\Om$.

%Let $f \in \ell^\infty(\Ga)$ and via the isomorphism of $C^*$-algebras
%of $\ell^\infty(\Ga)$ and $C(\beta\Ga)$ we write $\hat f$ for
%the unique extension of $f$ to $\beta\Ga$.

Let $\pi_e: \Om \to \{0,1\}$ denote the projection on the
$e$-component of $\Om$. Here $e$ is the neutral element of $\Ga$.
Let $\om_0=\ch_D$ be some fixed element of $\Om$ whose $\Ga$ orbit
is dense in $\Om$. Let $\psi : \ga \mapsto \ga\om_0$ be the
orbit map and let $\hat \psi$ denote its unique extension
to $\beta\Ga$. Thus $\hat\psi(q) = q\om_0$ for $q \in \beta\Ga$.
Finally recall that the semigroup product on $\beta\Ga$ is defined by
$$
A \in pq \iff \{\ga\in \Ga: \ga^{-1}A \in q\}\in p.
$$

Consider the map $L_p : \beta\Ga \to \beta\Ga$ and write
$\phi: \beta\Ga \to \{0,1\}$ for the map
$\phi=\pi_0 \circ \hat \psi \circ L_p$.
(Thus $\phi(q)=(pq\om_0)(e)$.) We define $J:\Om \to \Om$ by
$(J(\om))(\ga)=\om(\ga^{-1})$

We have
$$
\Qcal :=\phi^{-1}(1)=
\{q \in \beta\Ga: pq\om_0(e)=1\}.
$$
Now $pq\om_0(e)=1$ iff $e \in \chi^{-1}(pq\om_0)$ hence
\begin{align*}
\Qcal &=
\{q \in \beta\Ga: pq\om_0(e)=1\}\\&=
\{q \in \beta\Ga: e \in p\star q\om_0\} \\&=
\{q \in \beta\Ga: e \in p\star q\star D\} \\&=
\{q \in \beta\Ga:e \in\{\ga\in \Ga: \ga (q\star D)^{-1} \in p\}\}\\&=
\{q \in \beta\Ga:(q\star D)^{-1} \in p\}\\&=
\{q \in \beta\Ga: J(q\om_0) \in p\}.
\end{align*}
Thus
$(J\circ \hat\psi)(\Qcal)=p$ and since also
$(\hat\psi^{-1}\circ J^{-1})(p)=\Qcal$
we conclude that $\Qcal=\phi^{-1}(1)$ is not Borel measurable
in $\beta\Ga$. Finally, since also
$\Qcal = L_p^{-1}(\{q\in \beta\Ga:(q\om_0)(e)=1\})$
we see that $L_p$ is not Borel measurable.
\end{proof}

\br

In the next two lemmas let $\Om=\{0,1\}^\N$.
As above, we identify subsets $A$ of $\N$
with their characteristic functions $\ch_A\in \Om=\{0,1\}^\N$ and,
accordingly, filters on $\N$ with subsets of $\Om$.
Let $\phi: \Om \to \Om$ denote the ``flip" function defined by
$\phi(\om)_n = 1 - \om_n$.
We consider the measure space
$(\Om,\Sig_\la,\la)$, where $\Om=\{0,1\}^\N$, $\la$
is the Bernoulli measure
$\la=(\frac12(\del_0 +\del_1))^\N$,
and $\Sig_\la$ denotes the completion of the Borel $\sig$-algebra
with respect to $\la$. As usual we use the notation $\la_*$
and $\la^*$ for the induced inner and outer measures.

The assertions of the next lemma are
easily verified.

\begin{lem}\label{filters}
\begin{enumerate}
\item
The involution $\phi$ is measurable and it preserves $\la$.
\item
For $A \subset \N$ we have $\phi(\ch_A)=\ch_{A^c}$.
\item
If $\Fcal$ is a filter on $\N$ then $\phi\Fcal \cap\Fcal
=\emptyset$.
\item
If $\Fcal$ is a free filter on $\N$
(i.e. $\bigcap \Fcal =\emptyset$)
then, considered as a collection
of subsets of $\{0,1\}^\N$ it is a ``tail event", that is,
for every $m \in \N$,
$\Fcal= \{0,1\}^m \times \Fcal'$, with $\Fcal'
\subset \{0,1\}^\N$.
\item
A filter $\Fcal$ on $\N$ is an ultrafilter iff $\phi(\Fcal)
\cup \Fcal= \Om$.
\end{enumerate}
\end{lem}

\begin{lem}\label{m-filters}
Let $\Fcal$ be a free filter on $\N$. Then
\begin{enumerate}
\item
$\la_*(\Fcal) = 0$.
\item
$\la^*(\Fcal) \in\{0,1\}$.
\item
$\la^*(\Fcal) =1$ if $\Fcal$ is an ultrafilter.
\item
A free filter $\Fcal$ is measurable iff $\la^*(\Fcal)=0$,
and nonmeasurable iff  $\la^*(\Fcal)=1$. In particular,
every free ultrafilter is nonmeasurable.
\end{enumerate}
\end{lem}

\begin{proof}
If $\Fcal$ is a free filter on $\N$
and $\Gcal \subset \Fcal$ is a measurable tail event then it has
measure either $0$ or $1$. Thus $\la^*(\Fcal) \in \{0,1\}$.
This proves part 2. We also have
$\la_*(\Fcal) \in \{0,1\}$ and since
$\phi(\Fcal) \cap\Fcal =\emptyset$ it follows that
$$
1 = \la(\Om) \ge \la_*(\phi\Fcal) + \la_*(\Fcal)
= 2 \la_*(\Fcal).
$$
We conclude that $\la_*(\Fcal)=0$, proving part 1.
If $\Fcal$ is an ultrafilter then
$\Fcal \cup \phi\Fcal=\{0,1\}^\N$ and we conclude that
$$
1 = \la(\Om) \le \la^*(\phi\Fcal) + \la^*(\Fcal)
= 2 \la^*(\Fcal),
$$
whence $\la^*(\Fcal)=1$. This proves part 3. Part 4 is now clear.
\end{proof}

\br

\section{On the center of $\Ga^{\Dcal}$, the universal
distal Ellis group of $\Ga$}

Let $\Ga$ be a discrete abelian group. Let $\Dcal$
denote the closed $\Ga$-invariant subalgebra of (complex
valued) distal functions in $\ell^\infty(\Ga)$.
Let $D =\Ga^\Dcal=|\Dcal|$
denote the corresponding Gelfand space. It is well known that
$D$ is the largest right topological group compactification
of $\Ga$.

\begin{thm}\label{F}
Let $\Ga$ be an infinite discrete abelian group.
The topological center
of $D=\Ga^\Dcal$ is the same as the algebraic center and,
when $\Ga=\Z$, it also coincides with the canonical image of
$\Ga$ in $D$.
\end{thm}

\begin{proof}
In order to simplify our notation we will identify elements
of $\Ga$ with their images in $D$.
The coincidence of the topological and algebraic centers of $D$
is easy:
Suppose first that $p \in D$ is in the algebraic center of this group.
Then, as right multiplication is always continuous,
we have for any convergent net $q_\al \to q$ in $D$
$$
pq=qp=\lim q_\al p= \lim pq_\al,
$$
i.e. $L_p: D \to D$ is continuous.

Conversely, assume that $p$ is in the topological center;
i.e. $L_p: D \to D$ is continuous.
We note that if $q$ is an element of $D$ then $\ga q = q \ga$
for every $\ga \in G$. In fact choosing a convergent net
$\Ga\ni \ga_\al \to q$, by the commutativity of $\Ga$,
$$
\ga q = \ga \lim \ga_\al =
\lim \ga\ga_\al =\lim \ga_\al \ga= q \ga.
$$
Now, with this in mind, we have
$$
pq = p\lim \ga_\al = \lim p\ga_\al =
\lim \ga_\al p = qp,
$$
so that $p$ is indeed an element of the center.

Now to the more delicate task of showing that this center
coincides with $\Ga$. Let $p \in D$ be a central element.
If $p \not\in \Ga$ then there exists a {\it metric} minimal
distal dynamical system $(Y,\Ga)$ and a point $y_0\in Y$
such that
\begin{equation}\label{p}
py_0\not\in \Ga y_0.
\end{equation}
By assumption the map
$L_p:D \to D$ is continuous (in fact a homeomorphism)
and as we have seen it also commutes with every element
of $\Ga$. In other words, $L_p$ is an automorphism of
the system $(D,\Ga)$.
Now the dynamical system $(D,\Ga)$ is the universal distal
system and therefore, it admits a unique homomorphism
of dynamical systems $\hat\phi:(D,\Ga)\to (E(Y,\Ga),\Ga)$
onto the enveloping semigroup $E=E(Y,\Ga)$ (which by a theorem
of Ellis is in fact a group) such that $\hat\phi(e_D)=e_E$.
Now the map $\phi:p \mapsto \hat\phi(p)y_0$ (which we
write simply as $p \mapsto py_0$)
is a homomorphism $\phi: (D,\Ga) \to (Y,\Ga)$ with
$\phi(e)=y_0$. If $y_\al \to y$ is a convergent net
in $Y$ then there are $q_\al\in D$ with $y_\al=q_\al y_0$.
With no loss of generality we have $q_\al \to q$ in $D$,
so that in particular $y = \lim y_\al =\lim q_\al y_0=
qy_0$. Now we see that
\begin{align*}
py & = pqy_0  = (p \lim q_\al) y_0\\
& =(\lim pq_\al) y_0 = \lim p(q_\al y_0) \\
& = \lim p y_\al.
\end{align*}
Thus $p$ acts continuously on $Y$. Since also $p\ga=\ga p$
for every $\ga \in \Ga$ we conclude that $p$ is an automorphism
of the system $(Y,\Ga)$.

Note that this argument shows that {\it $p$ acts as an automorphism
of every factor of $(D,\Ga)$}. Therefore, our proof will be
complete when we find a minimal distal dynamical system
$(X,\Ga)$ extending $(Y,\Ga)$, say $\pi:(X,\Ga) \to (Y,\Ga)$,
where $p$ is not an automorphism.

\br

At this stage, in order to be able to use a method
of construction developed by Glasner and Weiss in \cite{GW},
we specialize to the case $\Ga=\Z$.
In particular the system $(Y,\Ga)$ which was
singled out in the above discussion has now the form
$(Y,T)$ where $T: Y \to Y$ is a self homeomorphism of $Y$
determined by the element $1 \in \Z$. Of course we can assume
that $Y$ is non-periodic (i.e. infinite).

The following construction is a special case of a general
setup designed in \cite{GW} for providing minimal extensions
of a given non-periodic minimal $\Z$-system $(Y,T)$.
We refer the reader to \cite{GW} for more details.

Set $X = Y \times K$ where $K$ denotes the circle group
$K=S^1=\{z\in \C: |z|=1\}$.
Let $\Tet$ be the family of continuous maps $\tet:Y \to K$.
For each $\tet \in \Tet$ let $G_\tet: X \to X$ be the map
$G_\tet(y,z)=(y,z\tet(y))$ and
$S_\tet=G_\tet^{-1}\circ (T\times {\id})\circ G_\tet$.
Thus
\begin{equation}\label{tet}
S_\tet :X \to X, \qquad
S_\tet(y,z) = (Ty,z\tet(y)\tet(Ty)^{-1}).
\end{equation}
Form the collection
$$
\Scal(T)=\{G_\tet^{-1}\circ (T\times {\id})\circ G_\tet
: \tet \in \Tet\}.
$$

Theorem 1 of \cite{GW} ensures that in the set $\overline{\Scal(T)}$
(closure with respect to the uniform convergence topology in
$\Homeo(X)$) there is a dense $G_\del$ subset $\Rcal$ such that
for every $R \in \Rcal$ the system $(X,R)$ is minimal, distal, and the
projection map $\pi:X \to Y$ is a homomorphism of dynamical systems
($\pi R(y,z)=T\pi(y,z)=Ty$).
Note that every $R \in \overline{\Scal(T)}$ has the form
$$
R=T_\phi :X \to X, \quad \text{where}\quad
T_\phi(y,z)= (Ty,z\phi(y)),
$$
for some continuous map $\phi: Y \to K$.
We will often use the fact that for $n \in \N$
the $n$-th iteration of $T_\phi$ has the form
\begin{equation}\label{cocycle}
T^n_\phi(y,z)= (T^n y,z\phi_n(y)),\quad{\text{where}}\quad
\phi_n(y) = \phi(T^{n-1}y)\cdots\phi(Ty)\phi(y).
\end{equation}
Note that when $\phi$ has the very special form
$\phi(y) = \tet(Ty)^{-1}\tet(y)$ for some continuous
$\tet:Y \to K$, the equation (\ref{cocycle}) collapses:
\begin{equation}\label{coboundary}
\phi_n(y) = \tet(T^n y)^{-1}\tet(y),
\quad{\text{hence}}\quad
S^n_\tet(y,z)= (T^n y, z\tet(T^n y)^{-1}\tet(y)).
\end{equation}

We temporarily fix an element $R=T_\phi \in \Rcal$.
As observed above, the element $p \in D$ defines an automorphism
of the system $(X,T_\phi)$; moreover we have for every
$x=(y,z) \in X$:
$$
\pi(px) = p\pi(x) = py.
$$
This last observation implies that $p:X \to X$ has the
form $p(y,z)=(py,\om(y,z))$ for some continuous map $\om:
Y\times K \to K$.

\begin{lem}
The function $\om$ has the form $\om(y,z)=z\psi(y)$ for some continuous
map $\psi:Y \to K$, whence
$$
p(y,z)=(py,z\psi(y))
$$
\end{lem}

\begin{proof}
There exists a net $\{n_\nu\}_{\nu\in I}$ in $\Z$ such that
$p = \lim n_\nu$ in $D$. Thus,
%denoting
%$$
%\phi_n(y) = \phi(T^{n-1}y)\cdots\phi(Ty)\phi(y),
%$$
%we have
for every $(y,z) \in X$
$$
p(y,z) = \lim T_\phi^{n_\nu}(y,z) =
\lim (T_\phi^{n_\nu}y,z\phi_{n_\nu}(y))=
(py,z\psi(y)),
$$
where the {\it point-wise} limit $\psi(y):= \lim \phi_{n_\nu}(y)$
is necessarily a continuous function.
\end{proof}

The commutation relation $pT_\phi=T_\phi p$ now reads:
\begin{align*}
pT_\phi(y,z) & = p(Ty,z\phi(y))=(pTy,z\phi(y)\psi(Ty))\\
& = T_\phi p(y,z) = T_\phi(py,z\psi(y))\\
& = (T p y,z\psi(y)\phi(py)).
\end{align*}
In turn this implies:
\begin{equation}\label{psi}
\phi(y)\psi(Ty)=\psi(y)\phi(py).
\end{equation}\label{n}
Similarly the commutation relations $pT^n_\phi=T^n_\phi p$ yield:
\begin{equation}\label{psi-n}
\phi_n(y)\psi(T^n y)=\psi(y)\phi_n(py).
\end{equation}

Next consider any sequence
$n_i \nearrow \infty$ such that
\begin{itemize}
\item
$\lim T^{n_i}y_0 =y_0$,
\item
$\lim \phi_{n_i}(y_0) = z'$, and
\item
$\lim \phi_{n_i}(py_0) = z''$.
\end{itemize}
%This can be done since by the minimality of $T_\phi$
%there is $n_i \nearrow \infty$ with $T^{n_i}_\phi
%(y_0,1)=(T^{n_i}y_0,\phi_{n_i}(y_0)) = (y_0,1)$.

Applying (\ref{psi-n}) and taking the limit as $i \to \infty$
we get $\psi(y_0)z'' = z'\psi(y_0)$, whence
necessarily also $z' = z''$.

The proof of Theorem \ref{F} will be complete when we
next show that for a residual subset $\Rcal_1$ of $\overline
{\Scal(T)}$, we have $\lim \phi_{n_i}(y_0) = z'
\ne  z'' = \lim \phi_{n_i}(py_0)$,
whenever $R=T_\phi \in \Rcal_1$. Then for any element $T_\phi
\in \Rcal \cap \Rcal_1$, $(X,T_\phi)$ will serve as a minimal
distal system where $p$ is not an automorphism.

\begin{prop}\label{prop}
For a given sequence $n_i \nearrow \infty$ with
$\lim T^{n_i}y_0 =y_0$, the set
$$
\Rcal_1 = \{T_\phi \in \overline{\Scal(T)}:\ \forall i\ \exists j>i,
\ |\phi_{n_j}(y_0) - \phi_{n_j}(py_0)| > 1\}
$$
is a residual subset of $\overline{\Scal(T)}$.
\end{prop}

%\begin{proof}
%For $i\in \N$ set
%$$
%E_i=\{T_\phi \in \overline{\Scal(T)}:\ \exists j>i,
%\ |\phi_{n_j}(y_0) - \phi_{n_j}(py_0)| > 1\}.
%$$
%Clearly $E_i$ is an open subset of $\overline{\Scal(T)}$
%and for $i < k$ we have $E_k \subset E_i$.

\begin{proof}
For $i\in \N$ and $\eta > 0$ set
$$
E_{i,\eta}=\{T_\phi \in \overline{\Scal(T)}:\ \exists j>i,
\ |\phi_{n_j}(y_0) - \phi_{n_j}(py_0)| > 1 + \eta\}.
$$
Clearly $E_{i,\eta}$ is an open subset of $\overline{\Scal(T)}$
and for $i < k$ we have $E_{k,\eta} \subset E_{i,\eta}$.

\begin{lem}\label{inv}
Given $i$ and  $\eta>0$,
for every $\tet_0\in \Tet$ there exists an $i_0 > i$ such that
$$
G^{-1}_{\tet_0} E_{i_0,\eta} G_{\tet_0} \subset E_{i_0,\eta/2}.
$$
\end{lem}

\begin{proof}
Fix $\tet_0 \in \Tet$. For sufficiently large $i_0$, for all $j>i_0$
the distances $d(T^{n_j}y_0,y_0)$ and $d(T^{n_j}py_0,py_0)$
are so small that
$$
|\tet(T^{n_j}y_0)^{-1}\tet(y_0)\phi_{n_j}(y_0)
- \tet(T^{n_j}py_0)^{-1}\tet(y_0)\phi_{n_j}(py_0)| > 1 + \eta/2
$$
holds whenever
$$
|\phi_{n_j}(y_0) - \phi_{n_j}(py_0)| > 1 + \eta.
$$
\end{proof}

%
%Moreover it is easily checked that for every $\tet \in \Tet$
%there exists an $i_0$ such that
%\begin{equation}\label{inv}
%G_\tet^{-1} \circ E_{i_0} \circ G_\tet \subset E_{i_0}.
%\end{equation}
%In fact,
%
%
%for sufficiently large $i_0$ for all $j>i_0$
%the distances $d(T^{n_j}y_0,y_0)$ and $d(T^{n_j}py_0,py_0)$
%are so small that
%$$
%|\tet(T^{n_j}y_0)^{-1}\tet(y_0)\phi_{n_j}(y_0)
%- \tet(T^{n_j}py_0)^{-1}\tet(y_0)\phi_{n_j}(py_0)| > 1,
%$$
%holds whenever
%$$
%|\phi_{n_j}(y_0) - \phi_{n_j}(py_0)| > 1.
%$$
%
%

%We will show that $E_i$ is also dense in $\overline{\Scal(T)}$.
%For this it suffices to show that
%$G^{-1}_\tet \circ (T \times {\id}) \circ G_\tet
%\in \overline{E_i}$ for every $\tet \in \Tet$, i.e.
%$T \times {\id} \in  G_\tet \overline{E_i}  G^{-1}_\tet$.
%
%Now for a fixed $\tet_0$ there is by Lemma \ref{inv},
%an $i_0 > i$ with
%$G_{\tet_0}^{-1}  E_{i_0}  G_{\tet_0} \subset E_{i_0}$,
%hence it suffices to show that
%$T \times {\id} \in \overline{E_{i_0}}$,
%since then
%$$
%T \times {\id} \in \overline{E_{i_0}} \subset
%G_{\tet_0}\overline{E_{i_0}} G_{\tet_0}^{-1}
%\subset
%G_{\tet_0}\overline{E_i} G_{\tet_0}^{-1}.
%$$
%Finally the next lemma will prove this last assertion
%and therefore also the density of $E_i$ for every $i$.

We will show that $E_{i,\eta}$ is also dense in $\overline{\Scal(T)}$.
For this it suffices to show that
$G^{-1}_\tet \circ (T \times {\id}) \circ G_\tet
\in \overline{E_{i,\eta}}$ for every $\tet \in \Tet$, i.e.
$T \times {\id} \in  G_\tet \overline{E_{i,\eta}}  G^{-1}_\tet$.

Now for a fixed $\tet_0$ there is by Lemma \ref{inv},
an $i_0 > i$ with
$G^{-1}_{\tet_0} E_{i_0,2\eta} G_{\tet_0} \subset E_{i_0,\eta}$,
hence it suffices to show that
$T \times {\id} \in \overline{E_{i_0,2\eta}}$,
since then
$$
T \times {\id} \in \overline{E_{i_0,2\eta}} \subset
G_{\tet_0}\overline{E_{i_0,\eta}} G_{\tet_0}^{-1}
\subset
G_{\tet_0}\overline{E_{i,\eta}} G_{\tet_0}^{-1}.
$$
Finally the next lemma will prove this last assertion
and therefore also the density of $E_{i,\eta}$ for every $i$
and $0 < \eta < 1$.

\begin{lem}
Given $i \in \N, 0 < \eta < 1$ and $\ep >0$ there exists
$\tet \in \Tet$ such that
\begin{enumerate}
\item
$d(T \times {\id}, G_\tet^{-1} \circ (T \times {\id}) \circ G_\tet)
< \ep$.
\item
$G_\tet^{-1} \circ (T \times {\id}) \circ G_\tet \in E_{i,\eta}$.
\end{enumerate}
\end{lem}

\begin{proof}
Let $I=[0,1]$ and set $h(0)=h(1/3)=h(2/3)=1, \ h(1) = -1$
and extend this function in an arbitrary way to a continuous
$h: I \to S^1$.
Choose $\del >0$ such that $|t - s|< \del$ implies
$|h(t)^{-1}h(s)-1| < \ep$.
Let $m \in \N$ be such that $2/m <\del$.
Let $U_1$ and $U_2$ be open neighborhoods of $y_0$ and $py_0$,
respectively, in $Y$ such that for $s=1,2$, the sets
$U_s, TU_s,\dots, T^{m-1}U_s$
are mutually disjoint. (Here we use the facts that $Y$ is
infinite and that $py_0 \not\in \{T^j y_0: j \in \Z\}$ (\ref{p}).)
Choose $k>i$ so that
$T^{n_k}y_0 \in U_1$ and $T^{n_k}py_0 \in U_2$.
Let $K_s \subset U_s,\ s=1,2$, be Cantor sets such that
$y_0, T^{n_k}y_0 \in K_1$ and $py_0, T^{n_k}py_0 \in K_2$.

Next define:
$$
g(y_0)=0, \ g(T^{n_k}y_0 )=1/3,\
g(py_0)=2/3, \ g(T^{n_k}py_0) = 1
$$
and extend this function in an arbitrary way to a continuous
function $g: K_1 \cup K_2 \to S^1$. We now extend $g$ to the
set $\cup_{j=0}^{m-1}T^j(K_1 \cup K_2)$ by setting
$g(y)=g(T^{-j}y)$ for $y \in T^j(K_1 \cup K_2)$.
Extend $g$ continuously over all of $Y$ in an arbitrary way.

Set
$$
\tilde{g}(y)= \frac1m \sum_{j=0}^{m-1} g(T^jy).
$$
Clearly $\tilde{g}\rest (K_1 \cup K_2) = g\rest (K_1 \cup K_2)$,
so that
$$
\tilde{g}(y_0)=0, \ \tilde{g}(T^{n_k}y_0 )=1/3,\
\tilde{g}(py_0)=2/3, \ \tilde{g}(T^{n_k}py_0) = 1.
$$
%
%$$
%\tilde{g}(y_0)=\tilde{g}(T^{n_k}y_0 )=\tilde{g}(py_0)=1,
% \quad {\text{and}}\quad \tilde{g}(T^{n_k}py_0) = -1.
%$$
Finally define $\tet: Y \to S^1$ by $\tet(y)=h(\tilde{g}(y))$.
Note that
\begin{equation}\label{values}
\tet(y_0)= \tet(T^{n_k}y_0 )=
\tet(py_0)=1, \quad {\text{and}}\quad \tet(T^{n_k}py_0) = -1.
\end{equation}

\br

Now
$$
G^{-1}_\tet\circ (T \times {\id}) \circ G_\tet (y,z)=
(Ty, z\tet(Ty)^{-1}\tet(y)) = (Ty, zh(\tilde{g}(Ty))^{-1}
h(\tilde{g}(y))).
$$
But
$$
|\tilde{g}(Ty) - \tilde{g}(y)|< 2/m < \del,
$$
hence $|h(\tilde{g}(Ty))^{-1}h(\tilde{g}(y))-1| < \ep$
and therefore also
$$
d(T \times {\id}, G_\tet^{-1} \circ (T \times {\id}) \circ G_\tet)
< \ep.
$$

\br

This proves part (1) of the lemma and we now proceed to prove
part (2).
We have to show that
$G_\tet^{-1} \circ (T \times {\id}) \circ G_\tet \in E_i$.
But this map has the form
$$
S_\tet :X \to X, \qquad
S_\tet(y,z)= G_\tet^{-1}\circ (T\times {\id})
\circ G_\tet = (Ty,z\tet(Ty)^{-1}\tet(y)),
$$
so that, by (\ref{coboundary}), we have to show that
there exists $j > i$ with
$$
\ |\tet(T^{n_j}y_0)^{-1}\tet(y_0) -
\tet(T^{n_j}py_0)^{-1}\tet(py_0)| > 1 + \eta.
$$
Since, by the choice of $\tet$ (\ref{values}), we have
$$
\ |\tet(T^{n_k}y_0)^{-1}\tet(y_0) -
\tet(T^{n_k}py_0)^{-1}\tet(py_0)|
= | 1 - (-1)|=2 > 1 + \eta,
$$
this completes the proof of the lemma.
\end{proof}

To conclude the proof of Proposition \ref{prop}
observe that, for instance, the dense $G_\del$ set
$\bigcap_{i=1}^\infty E_{i,1/2}$ is contained in $\Rcal_1$.
\end{proof}

This also concludes the proof of Theorem \ref{F}.

\end{proof}

%Let $\mu$ be a $T$-invariant measure on $Y$ and let
%$\{e_i(y)\}_{i=0}^\infty$, with $e_0(y)\equiv 1$,
%be an orthonormal basis for $L^2(Y,\mu)$.
%Let
%$$
%\om(y,z)=\sum_{j=0}^\infty\sum_{i=0}^\infty c_{ij}e_i(y)z^j
%$$
%be the unique expansion of $\om(y,z)$ in $L^2(X,\mu\times\la)$,
%where $\la$ is Lebesgue measure on $K=S^1$.
%By (\ref{omega}) we now have:
%$$
% \sum_{j=0}^\infty\sum_{i=0}^\infty c_{ij}e_i(Ty)(z\phi(y))^j =
% \left(\sum_{j=0}^\infty\sum_{i=1}^\infty c_{ij}e_i(y)z^j\right)
%\cdot\phi(py),
%$$
%whence
%$$
%\sum_{i=0}^\infty c_{ij}e_i(Ty)\phi(y)^j =
%\sum_{i=0}^\infty c_{ij}e_i(y)\phi(py),
%$$
%for each $j$,

%and finally, by uniqueness,
%$$
%c_{ij}\phi(y)^j = c_{ij}\phi(py)
%$$
%for every $i$ and $j$.
%
%
%For $j = 0$, if $c_{ij}\not = 0$ we get $\phi(py)\equiv 1$,
%but this contradicts the minimality of $T_\phi$. For $j\not\in
%\{0,1\}$ if $c_{ij}\not = 0$ we get $\phi(y)^j = z^j\phi(py)$,
%hence $z^j \equiv \phi(y)^j \phi(py)^{-1}$ which is absurd.
%Thus we conclude that $c_{ij} = 0$ for $j\ne 1$.
%Finally when $j=1$ we get $c_{i1}\phi(y) = c_{i1}z\phi(py)$

%, so that
%$$
%\om(y,z) = \sum_{i=1}^\infty c_{i1}e_i(y) z
%= \rho(y) z,
%$$
%with $\rho(y) = \sum_{i=1}^\infty c_{i1}e_i(y)$, as required.
%\end{proof}
%
%Now the relation (\ref{omega}) reads
%$$
%\om(Ty,z\phi(y)) = \rho(Ty)z\phi(y) =
%\om(y,z)\phi(py) = \rho(y) z \phi(py),
%$$
%hence

\section{Addendum: On the image of Borel sets}
%\title[On the image of Borel sets]
%{On the image of Borel sets.
%\\
%{\tiny{An addendum to
%``On two problems concerning topological centers", Topology Proc. 33 (2009), 29--39.}}}
%
%\author{Eli Glasner}
%
%\address{Department of Mathematics\\
%     Tel Aviv University\\
%         Tel Aviv\\
%         Israel}
%\email{glasner@math.tau.ac.il}
%
%\address {Institute of Mathematics\\
% Hebrew University of Jerusalem\\
%Jerusalem\\
% Israel}
%\email{weiss@math.huji.ac.il}
%\email{iftahday@math.tau.ac.il}         
%\email{glasner@math.tau.ac.il}
% \urladdr{http://www.math.tau.ac.il/$^\sim$glasner}

%
%\date{September 5, 2016}
%
%\begin{abstract}
%We prove a lemma necessary for closing a gap in the proof of a theorem in the paper 
%referred to in the title of this note.
%\end{abstract}
%

%\subjclass[2000]{47A16, 47A35, 37A05, 37A25, 37B05, 37B20, 54H20, 60G15}

\keywords{}

%\thanks{This research was partially supported by Grant No 2006119
%from the United States-Israel Binational Science Foundation (BSF)}

%\thanks{This research was supported by a grant of ISF 668/13}

%\thanks{{This work is supported by ISF grant 
%\#1157/08}}

%\maketitle
%\tableofcontents

I am indebted to Professors Neil Hindman and Dona Strauss who pointed out to me 
a gap in the proof of Theorem 2.1.
% in \cite{G}.
%The last sentence of this proof, as follows.
Since the compact Hausdorff space $\beta\Ga$ is not Polish, in order to justify the claim
(in the last sentence of the proof):
\begin{quote}
Finally, since also 
$Q = L_p^{-1}(\{q \in \beta \Ga : (q\om_0)(e) = 1\})$ we see that 
$L_p$ is not Borel measurable,
\end{quote}
one needs an additional lemma, which I provide below.

\br

\begin{lem}
Let $X$ be a compact Hausdorff space.
Let $B \subset X$ be a Borel set (i.e. a member of the smallest $\sig$-algebra
containing the open sets).
Let $f : X \to C$ be a continuous surjection, where $C$ denotes the Cantor set.
Suppose also that $B = \{x \in X : f(x) \in f(B)\} = f^{-1}(f(B))$.
Then, $f(B)$ is universally measurable.
\end{lem}

\begin{proof}
Let $\mu$ be a probability measure on $C$.
Denote by $\mu^*$ the corresponding outer-measure on $C$ and
recall that the restriction of $\mu^*$ to $\Bcal_\mu$, the completion of the $\sig$-algebra of 
Borel sets with respect to $\mu$, is a measure.

As $X$ is a compact Hausdorff space 
there exists a regular probability measure $\nu$ on $X$
which extends $\mu$ in the sense that $\nu(f^{-1}(A)) =\mu(A)$
for every Borel set $A \subset C$.
(Here we use the Hahn-Banach and the Riesz representation theorems,
see e.g. Rudin's Real and complex analysis, Theorem 2.14.)

Write $B^c = X \setminus B$ and note that
$f(B^c) = f(B)^c = C \setminus f(B)$.

If either $\mu^*(f(B))=0$ or $\mu^*(f(B^c))=0$ then $f(B)$
is $\mu^*$-measurable and we are done. So we assume that
both outer measures are positive.

Given $\ep > 0$, there is a compact set $K \subset B$ with $\nu(B \setminus K) < \ep$. 
Let $\tilde{K} = f(K) \subset f(B)$, a compact subset of $f(B)$ with
$K \subset f^{-1}(\tilde{K}) \subset B$.

Also, 
there is a compact set $L \subset B^c$
with $\nu(B^c \setminus L) < \ep$. Let $\tilde{L} = f(L) \subset f(B^c)$, 
a compact subset of $f(B^c)$ with
$L \subset f^{-1}(\tilde{L}) \subset B^c$.
 
Now we have
$$
\mu(\tilde{K} \cup \tilde{L}) = \mu(K \cup L)  > 1 -  2 \ep.
$$
As $\ep$ is arbitrary this implies that $f(B)$ is $\mu^*$-measurable.
\end{proof}

%
%\br
%
%\bibliographystyle{amsplain}
%\begin{thebibliography}{10}
%
%\bibitem{G}
%Eli Glasner,
%{\em On two problems concerning topological centers},
%Topology Proc. {\bf 33}, (2009), 29--39.
%
%\end{thebibliography}
%
%

\end{document}